\documentclass[graybox]{svmult}
\pdfoutput=1  

\usepackage{type1cm}        
\usepackage{makeidx}         
\usepackage{graphicx}        
\usepackage{multicol}        
\usepackage[bottom]{footmisc}

\usepackage{newtxtext}       %
\usepackage{newtxmath}       

\usepackage{wrapfig}
\usepackage{mathtools}
\newcommand{\R}{\mathbb{R}}
\newcommand{\T}{\mathcal{T}}
\DeclareMathOperator{\conv}{conv}

\graphicspath{{./graphics}}

\begin{document}

\title*{Construction of 4D Simplex Space-Time Meshes for Local Bisection Schemes}
\author{David Lenz}
\institute{David Lenz \at Argonne National Laboratory, Lemont, IL, USA, \email{dlenz@anl.gov}
}
%
%
\maketitle

\abstract*{Adaptive mesh refinement is a key component of efficient unstructured space-time finite element methods. Underlying any adaptive mesh refinement scheme is, of course, a method for local refinement of simplices. However, simplex bisection algorithms in dimension greater than three have strict mesh preconditions which can be hard to satisfy. We prove that certain four-dimensional simplex space-time meshes can be handled with a relaxed precondition. Namely, we prove that if a tetrahedral mesh is 4-colorable, then we can produce a 4D simplex mesh which always satisfies the bisection precondition. We also briefly discuss strategies to handle tetrahedral meshes which are not 4-colorable.}

\section{Introduction}
\label{sec:intro}
Space-time finite element methods (FEMs) approximate the solution to a PDE `all-at-once' in the sense that a solution is produced at all times of interest simultaneously. This is achieved by treating time as just another variable and discretizing the entire space-time domain with finite elements. Naturally, discretizing the entire space-time domain creates a linear system with many more degrees of freedom (DOFs) than discretizing just the spatial domain. 
Adaptive mesh refinement can produce space-time discretizations that yield accurate solutions with relatively few degrees of freedom. For example, Langer \& Schafelner~\cite{langer2020} compared space-time FEMs using uniformly  and adaptively refined meshes; they found that obtaining the same approximation error for their tests required more degrees of freedom on the uniform meshes by one to two orders of magnitude.

A technical challenge in the implementation of adaptive mesh refinement is, of course, the mesh refinement scheme. In order to refine a geometric element while introducing as few new DOFs as possible, algorithms typically employ a ``red-green'' approach (uniform refinement with closure) or an element bisection approach. Here we focus on bisection schemes in four dimensions but note that work on arbitrary-dimensional red-green schemes was recently undertaken by Grande~\cite{grande2019}. Stevenson~\cite{stevenson2008} has studied a bisection algorithm for arbitrary-dimensional simplicial meshes, which is a good candidate for four-dimensional space-time meshes. However, the algorithm relies on a mesh precondition that is difficult to satisfy.

In this article, we weaken this strict precondition for certain space-time meshes. We prove that the precondition on four-dimensional simplex meshes can be reduced to a precondition on an underlying three-dimensional mesh. This means that the condition needs only be checked on a much smaller mesh. In addition, if one cannot immediately verify the precondition, refinements to this smaller three-dimensional mesh can be made to automatically satisfy the precondition.

In Section 2, we describe a particular method for creating four-dimensional space-time meshes. Meshes of this type have a  structure which will be exploited in Section 3, where we summarize key concepts from Stevenson's bisection algorithm and then prove our main result. Finally, we conclude with some remarks on how this relaxed precondition can be used in practice.

\section{Four-Dimensional Space-Time Mesh Construction}
\label{sec:mesh-construction}
When applying space-time FEMs to solve a PDE, it is generally necessary to create a space-time mesh that corresponds to a given spatial domain. A convenient method for doing so is to repeatedly extrude spatial elements (typically triangles or tetrahedra) into higher-dimensional space-time prisms and then subdivide these prisms into simplicial space-time elements (tetrahedra or pentatopes). We refer to mesh generation methods of this type as \emph{extrusion-subdivision} schemes. 

The method of extrusion-subdivision has appeared in several places in recent years. This idea was applied to moving meshes by Karabelas \& Neum\"{u}ller~\cite{karabelas2015} and is discussed in a report by Voronin~\cite{voronin2018}, where it is described in the context of an extension to the MFEM library~\cite{mfem-library}. For stationary (non-moving) domains, the extrusion step is straightforward, but subdividing the space-time prisms can be done in several ways. Behr~\cite{behr2008} describes a method for subdividing space-time prisms using Delaunay triangulations, while subdivision based on vertex orderings is used in \cite{karabelas2015}, \cite{voronin2018}.

In this paper, we consider space-time meshes produced by extrusion-subdivision where prism subdivision is defined in terms of vertex labels. In particular, we will assume that a $k$-coloring has been imposed on the mesh; that is, each vertex in the mesh has one of $k$ labels (colors) attached to it and no two vertices connected by an edge share the same label. 

Before we describe our particular prism subdivision method, we need to establish some notation. Suppose $\vec{a} = (a_0, a_1, \ldots, a_{d-1}) \in \R^d$. For any $r \in \R$, we define the map $\psi_r: \R^d \to \R^{d+1}$ by
\begin{equation}
\psi_r(\vec{a}) = (a_0, a_1, \ldots, a_{d-1}, r).
\end{equation}
For a simplex $T = \conv(\vec{a}, \vec{b}, \vec{c}, \vec{d})$ (here $\vec{a}, \vec{b}, \vec{c}, \vec{d} \in \R^d$), we define
\begin{equation}
    \psi_r(T) = \conv(\psi_r(\vec{a}), \psi_r(\vec{b}), \psi_r(\vec{c}), \psi_r(\vec{d}));
\end{equation}
that is, $\psi_r(T)$ is the embedding of $T$ into $\R^{d+1}$ within the plane $x_{d+1} = r$.

Next, the extrusion operator $\Phi_{r,s}$ produces the set of all points between $\psi_r(T)$ and $\psi_s(T)$. This is the convex hull of points in $\psi_s(T)$ and $\psi_r(T)$, which is a right tetrahedral prism:
\begin{equation}
    \Phi_{r,s}(T) = \conv(\psi_r(T), \psi_s(T)).
\end{equation}
We make one final notational definition to declutter the following exposition. Given a series of real values $\mathcal{S} = \{s_0, \ldots, s_M\}$, we define
\begin{equation}\label{eq:time-slice}
\Phi^\mathcal{S}_i(T) = \Phi_{s_i, s_{i+1}}(T) \quad \text{where } s_j \in \mathcal{S}, \text{ for } 0 \leq j \leq M-1.
\end{equation}
We refer to $\mathcal{S}$ as a collection of ``time-slices,'' which determine how spatial elements are extruded into space-time prisms. The value $s_j$ is thus the ``$j^{th}$ time-slice.'' For most problems, the set of initial time-slices is fixed ahead of time, so it is often convenient to omit the superscript $\mathcal{S}$. We adopt this shorthand for the remainder of this paper. For an illustration of the operators $\psi$ and $\Phi$ in two spatial dimensions, see Figure \ref{fig:extrusion}.

\begin{figure}
    \centering
    \includegraphics[width=.7\textwidth]{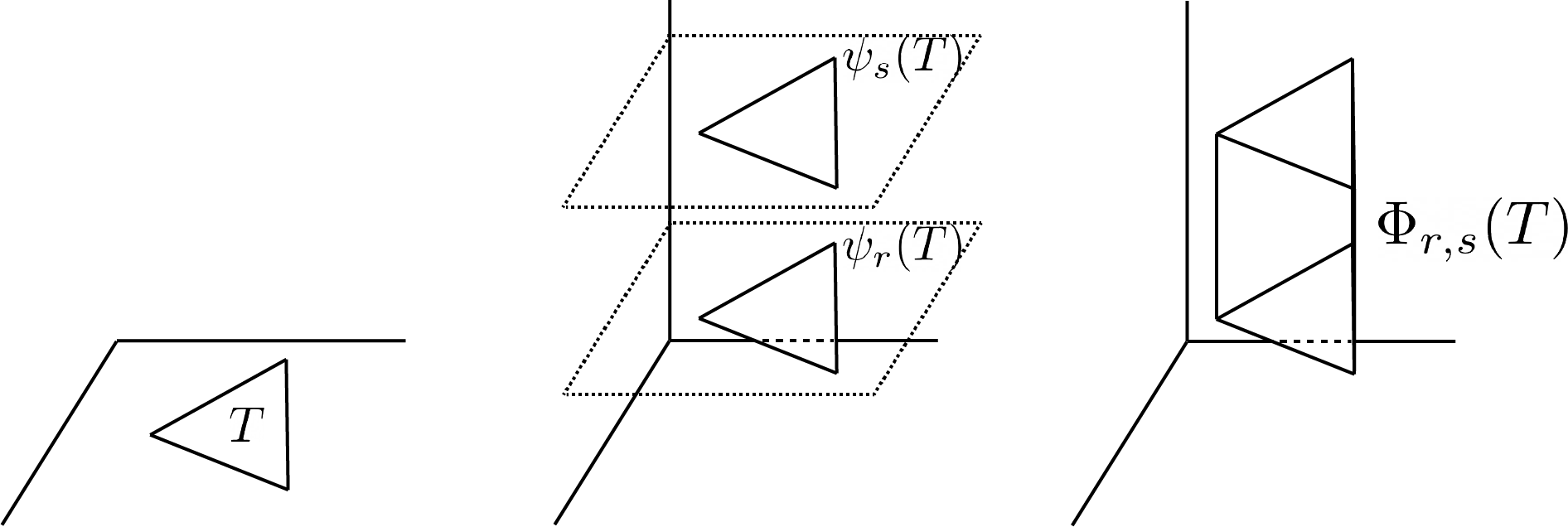}
    \caption{Extrusion of a 2D simplex into a 3D simplex prism. Left: The spatial element. Center: Copies of the spatial element embedded in space-time. Right: The space-time prism element.}
    \label{fig:extrusion}
\end{figure}

For the remainder of this article, we will focus on the case $d=3$; that is, problems in four-dimensional space-time.\footnote{
    Due to the inherent difficulty in visualizing four-dimensional objects on a two-dimensional page, we will continue to illustrate figures in three-dimensional space-time. These figures are meant only as a guide for the reader to develop some geometric intuition.
}
Let $\T$ be a conforming tetrahedral mesh with a 4-coloring,\footnote{
    Not every tetrahedral mesh admits a 4-coloring, although all are 5-colorable. We discuss how to handle meshes which are not 4-colorable at the end of Section 3.
}
and let the symbols $A, B, C, D$ denote the four labels to be associated with each vertex of $\T$. Let $\mathcal{S}$ be a set of time-slices which define the extrusion of spatial elements into space-time prisms (c.f. equation \ref{eq:time-slice}). We will use the following rule to subdivide the tetrahedral prisms formed by extruding elements of $\T$.

\begin{figure}
    \sidecaption
    \includegraphics[width=.25\textwidth]{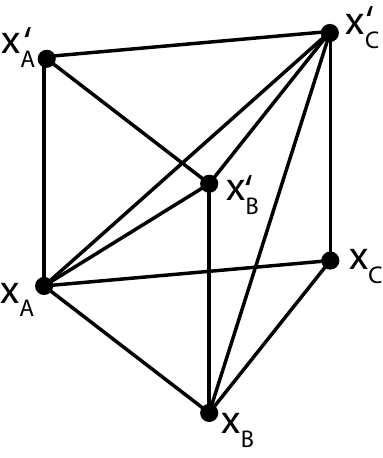}
    \caption{Space-time triangular prism subdivision based on a coloring of the underlying spatial mesh.}
    \label{fig:subdivision}
\end{figure}

\begin{definition}[Subdivision Rule]\label{def:subdivision-rule}
Let $T \in \T$ be a tetrahedron with vertices $v_A, v_B, v_C, v_D$, where the subscript of each vertex denotes its color label. For a given time slice $s_i \in \mathcal{S}$, let $x_A = \psi_{s_i}(v_A)$ and $x_A' = \psi_{s_{i+1}(v_A)}$ (and likewise for B, C, D). The rule for subdividing $\Phi_i(T)$ is to create the pentatopes:
\begin{equation}
    \begin{aligned}
    \tau_1 &= \conv(x_A, x_B, x_C, x_D, x_D')\\
    \tau_2 &= \conv(x_A, x_B, x_C, x_C', x_D')\\
    \tau_3 &= \conv(x_A, x_B, x_B', x_C', x_D')\\
    \tau_4 &= \conv(x_A, x_A', x_B', x_C', x_D')
    \end{aligned}
\end{equation}
\end{definition}
An illustration of this subdivision rule is given in figure \ref{fig:subdivision}. Because the labeling of each vertex is shared by all elements, this subdivision scheme always produces a conforming mesh of pentatopes. However, we omit the detailed proof because of space constraints.
 
\section{Conforming Bisection of Space-Time Simplicial Elements}
The aim of this section is to outline Stevenson's bisection algorithm~\cite{stevenson2008}\footnote{
    The bisection rule studied by Stevenson has  also been studied by Maubach~\cite{maubach1995} and Traxler~\cite{traxler1997}.
}, 
with particular attention on the mesh precondition for its validity. Then, we show that space-time meshes produced by extrusion-subdivision (using definition \ref{def:subdivision-rule}) will always meet this precondition if the underlying spatial mesh is 4-colorable. The upshot of this is that all of the work to make an admissible mesh can be done in three dimensions instead of four. This is especially useful in light of the fact that there are many more meshing utilities for three-dimensional domains than four-dimensional domains.

The following definitions are due to Stevenson~\cite{stevenson2008}.

\begin{definition}\label{def:tagged-simplex}
A \emph{tagged pentatope} $t$ is an ordering of the vertices of some pentatope $\tau = \conv(x_0, x_1,x_2, x_3, x_4)$, together with an integer $0 \leq \gamma \leq 3$ called the \emph{type}. We write
\begin{equation}
    t = (x_0, x_1, x_2, x_3, x_4)_\gamma
\end{equation}
to denote this ordering-type pair.
\end{definition}

\begin{definition}\label{def:reflection}
The \emph{reflection} of a tagged pentatope $t$ is another tagged pentatope $t_R$ such that the bisection rule produces the same child pentatopes for $t$ and $t_R$. The unique reflection of $t = (x_0, x_1, x_2, x_3, x_4)_\gamma$ is 
\begin{equation}
t_R = \begin{dcases}
(x_4, x_3, x_2, x_1, x_0)_\gamma &\text{if } \gamma = 0\\
(x_4, x_1, x_3, x_2, x_0)_\gamma &\text{if } \gamma = 1\\
(x_4, x_1, x_2, x_3, x_0)_\gamma &\text{if } \gamma = 2,3
\end{dcases}.
\end{equation}
\end{definition}

\begin{definition}\label{def:reflected-neighbors}
Two tagged pentatopes $t$ and $t'$ are \emph{reflected neighbors} if they share a common hyperface, have the same integer type, and the vertex order of $t'$ matches the vertex order of $t$ or $t_R$ in all but one position.
\end{definition}

The notion of reflected neighbors is critical to our proof of the main result, so it is worth illustrating the concept with some examples.
\begin{example}
    Let $t = (x_0, x_1, x_2, x_3, y)_1$ and $t' = (z, x_0, x_1, x_2, x_3)_1$. Then $t$ and $t'$ are NOT reflected neighbors. Although the relative ordering of their shared vertices is consistent, the taggings differ in every position.
\end{example}
\begin{example}
    Let $t = (x_0, z, x_1, x_2, x_3)_1$ and $t' = (x_3, y, x_2, x_1, x_0)_1$. Then $t$ and $t'$ are reflected neighbors. To see this, we note that $t_R = (x_3, z, x_2, x_1, x_0)_1$; thus $t_R$ and $t'$ differ only in the second position. Likewise, we could have shown that $t_R'$ and $t$ differ on at most one position.
\end{example}

\begin{definition}\label{def:consistent-tagging}
The tagging of two pentatopes $t = (x_0, x_1, x_2, x_3, x_4)_\gamma$ and $t' = (x_0', x_1', x_2', x_3', x_4')_\gamma$ which share a hyperface is said to be \emph{consistent} if the following condition is met:
\begin{enumerate}
\item If $\overline{x_0,x_4}$ or $\overline{x_0', x_4'}$ is contained in the shared hyperface, then $t$ and $t'$ are reflected neighbors (N.B. these are the bisection edge for each element; see definition \ref{def:bisection}).
\item Otherwise, the two children of $t$ and $t'$ which share the common hyperface are reflected neighbors.
\end{enumerate}
A \emph{consistent tagging of a mesh} is a tag for each element in the mesh such that any two neighboring elements are consistently tagged.
\end{definition}

In essence, Definition \ref{def:consistent-tagging} states that in a consistently tagged mesh, any pair of neighboring elements are either reflected neighbors or they do not share a common refinement edge. Furthermore, when two elements do not share a common refinement edge, their adjacent children will be reflected neighbors after one round of bisection.

\begin{definition}[Bisection Rule]\label{def:bisection}
Given a tagged pentatope $t = (x_0, x_1, x_2, x_3, x_4)_\gamma$, applying the bisection rule produces the children:
\begin{equation}
    t_1 = (x_0, x', x_1, x_2, x_3)_{\gamma'}
    \qquad 
    t_2 =
    \begin{dcases}
        (x_4, x', x_3, x_2, x_1)_{\gamma'} &\text{if } \gamma = 0\\
        (x_4, x', x_1, x_3, x_2)_{\gamma'} &\text{if } \gamma = 1\\
        (x_4, x', x_1, x_2, x_3)_{\gamma'} &\text{if } \gamma = 2,3
    \end{dcases},
\end{equation}
where $x' = (x_0 + x_4)/2$ and $\gamma' = (\gamma + 1)\mod{4}$. We say that edge $\overline{x_0x_4}$ is the \emph{refinement edge}.
\end{definition}

We can now state the main result of this section.

\begin{proposition}\label{prop:consistent-tagging}
Let $\mathcal{T} \subset \R^3$ be a 4-colorable tetrahedral mesh and $\mathcal{T}' \subset \R^4$ be the pentatope mesh produced by extrusion-subdivision according to definition \ref{def:subdivision-rule}. Then $\mathcal{T}'$ admits a consistent tagging.
\end{proposition}
\begin{proof}
We will prove the result by constructing a consistent tagging of $\T'$ directly from a 4-coloring of $\T$. 

To show that a tagging of $\T'$ is consistent, it suffices to consider an arbitrary element of $\T'$ and show that each of its neighbors satisfy the conditions in Definition \ref{def:consistent-tagging}. Let $\tau \in \T'$ be an abitrary element. Since $\T$ is created by extrusion-subdivision, there is some time-slice $s_i$ and $T \in \T$ such that $\tau \subset \Phi_i(T)$. Since $\Phi_i(T)$ was subdivided according to definition \ref{def:subdivision-rule}, each resulting pentatope is defined explicitly by the 4-coloring on $\T$.

For the above four pentatopes in definition \ref{def:subdivision-rule}, we make the following tagging (in each case $t_i$ is a tagging of $\tau_i$):
\begin{equation}\label{eq:tagging}
\begin{aligned}
t_1 = (x_D, x_C, x_B, x_A, x_D')_0
&\qquad t_2 = (x_C, x_B, x_A, x_D', x_C')_0\\
t_3 = (x_B, x_A, x_D', x_C', x_B')_0
&\qquad t_4 = (x_A, x_D', x_C', x_B', x_A')_0
\end{aligned}
\end{equation}

We will show that each neighbor $\tau'$ satisfies the consistency condition in Definition \ref{def:consistent-tagging}. There are three cases:
\begin{enumerate}
    \item $\tau$ and $\tau'$ are both pentatopes within the same space-time prism.
    \item $\tau$ and $\tau'$ belong to different space-time prisms extruded from the same spatial element; for instance, $\tau \subset \Phi_i(T)$ and $\tau' \subset \Phi_{i+1}(T)$.
    \item $\tau$ and $\tau'$ belong to different space-time prisms within the same space-time slab; for instance, $\tau \subset \Phi_i(T)$ and $\tau' \subset \Phi_i(T')$.
\end{enumerate}

Consider Case (1). Since $\tau$ and $\tau'$ are neighbors, they must be a pair $\tau_i, \tau_{i+1}$ ($j=1,2,3$), since only consecutive pairs in our list are neighbors. In this case, the adjacent children of each pentatope are reflected neighbors. To make this point explicit, consider the neighboring elements $\tau_1, \tau_2$. The children formed by bisecting these pentatopes are:
\begin{equation}
t_1 \to
\begin{dcases}
    (x_D, z_1, x_C, x_B, x_A)_1\\
    (x_D', z_1, x_A, x_B, x_C)_1
\end{dcases}
\qquad 
t_2 \to
\begin{dcases}
    (x_C, z_2, x_B, x_A, x_D')_1\\
    (x_C', z_2, x_D', x_A, x_B)_1
\end{dcases}
\end{equation}
where $z_i$ are the new midpoints of the bisected edges.
From here we note that the second child of $t_1$ is the reflected neighbor of $t_2$. The same exercise shows that the pairs $t_2,t_3$ and $t_3,t_4$ also share this property. Thus all pairs of neighbors in Case (1) are consistently tagged.

In Case (2), $\tau$ and $\tau'$ are neighboring pentatopes belonging to different space-time slabs; without loss of generality, $\tau \subset \Phi_i(T)$ and $\tau' \subset \Phi_{i+1}(T)$ for some $T \in \T$. With two space-time prisms, we have three sets of vertices at three different time-slices. We denote vertices in the highest (latest) time hyperplane with double primes (''), vertices in the middle hyperplane with single primes ('), and vertices in the lowest (earliest) hyperplane with no primes; see Figure \ref{fig:double-extrusion}.

\begin{figure}
    \sidecaption
    \includegraphics[width=.3\textwidth]{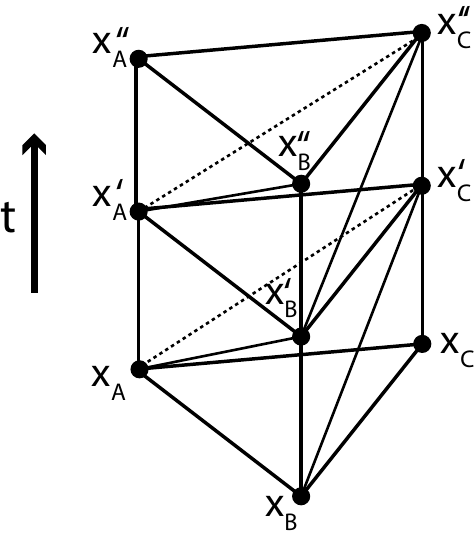}
    \caption{Illustration of vertex labeling when considering consecutive space-time time prisms in two spatial dimensions.}
    \label{fig:double-extrusion}
\end{figure}

Since $\tau$ and $\tau'$ belong to consecutive timeslices, the shared hyperface between the two must be $\conv(x_A', x_B', x_C', x_D')$. Thus $\tau = \conv(x_A, x_A', x_B', x_C', x_D')$ and $\tau' = \conv(x_A', x_B', x_C', x_D', x_D'')$. According to the tagging scheme described above, the tags on these two pentatopes are
\begin{equation}
t = (x_A, x_D', x_C', x_B', x_A')_0 \quad t' = (x_D', x_C', x_B', x_A', x_D'')_0,
\end{equation}
and thus their child elements are
\begin{equation}
t \to
\begin{dcases}
    (x_A, z, x_D', x_C', x_B')_1\\
    (x_A', z, x_B', x_C', x_D')_1
\end{dcases}
\qquad 
t' \to
\begin{dcases}
    (x_D', z', x_C', x_B', x_A')_1\\
    (x_D'', z', x_A', x_B', x_C')_1
\end{dcases}.
\end{equation}
where $z$ and $z'$ are the new vertices created by bisecting $\tau$ and $\tau'$, respectively. By inspection, the second child of $t$ and the first child of $t'$ are reflected neighbors; hence the tags $t, t'$ are consistent.

Finally, consider Case (3). Since $\T'$ is conforming, when $\tau \subset \Phi_i(T)$ and $\tau' \subset \Phi_i(T')$ they must share a vertical edge like $\overline{x_Ax_A'}$, which is always a bisection edge. Since the vertex labels are ``global'' labels, both $\tau$ and $\tau'$ must agree on the order in which the labeled vertices appear. Furthermore, $\tau$ and $\tau'$ share all but one vertex in common. Since all but one vertex is shared, both pentatopes have the same labels for the same vertices, and vertex order is uniquely determined by vertex label, the vertex orders agree on all but one position. Hence the tagged pentatopes are reflected neighbors.

Thus all three cases result in consistent taggings. Therefore, the tagging defined by equation \ref{eq:tagging} is consistent.
\end{proof}

The critical piece of this proof is the tagging scheme defined in equation \ref{eq:tagging}. This precise ordering of vertices is needed for the rest of the proof to work. Furthermore, since the vertex orders are determined by the 4-coloring on $\T$, a consistent tagging of the space-time mesh can be constructed in linear time.

\begin{corollary}
Let $\mathcal{T}$ and $\mathcal{T}'$ be as in Proposition \ref{prop:consistent-tagging}. Given a 4-coloring on $\mathcal{T}$, the space-time mesh $\mathcal{T}'$ can be consistently tagged in $O(N)$ time, where $N$ is the number of vertices in $\mathcal{T}'$.
\end{corollary}

Not every tetrahedral mesh is 4-colorable, but this can be worked around. First, we note that regular tetrahedral meshes are 4-colorable, so for rectilinear domains one can start with a coarse regular mesh and bisect until a desired resolution is met. In addition, Traxler~\cite{traxler1997} has shown (indirectly) that tetrahedral meshes over simply connected domains are 4-colorable iff every edge is incident to an even number of tetrahedra.

Finally, any tetrahedral mesh can be made 4-colorable by dividing each element via barycentric subdivision and then choosing the following colors: each vertex of the original mesh is colored A, the new center of each edge is colored B, the new center of each face is colored C, and the new center of each tetrahedron is colored D. Barycentric subdivision creates new elements with one of each kind of point, which means that this is indeed a 4-coloring.

\section{Conclusions}
We described a method for creating four-dimensional simplex space-time meshes from a given spatial mesh which has a 4-coloring. This procedure was based on the general extrusion-subdivision framework, with a new subdivision rule which is defined in terms of vertex labels (colors). We then proved that meshes of this form always satisfy the strict precondition of Stevenson's bisection algorithm, which can be used to adaptively refine space-time meshes. Finally, we showed that even when a tetrahedral mesh is not 4-colorable, the barycentric subdivision of the mesh will be.

\begin{acknowledgement}
    This work is supported by the
    U.S. Department of Energy, Office of Science, Advanced Scientific Computing Research under Contract DE-AC02-06CH11357, and the Exascale
    Computing Project (Contract No. 17-SC-20-SC), a collaborative effort of the U.S. Department of Energy Office of Science and the National Nuclear Security Administration.
\end{acknowledgement}

\bibliographystyle{spmpsci}
\bibliography{stbisect}

\pagebreak
The submitted manuscript has been created by UChicago Argonne, LLC, Operator of Argonne National Laboratory ("Argonne”). Argonne, a U.S. Department of Energy Office of Science laboratory, is operated under Contract No. DE-AC02-06CH11357. The U.S. Government retains for itself, and others acting on its behalf, a paid-up nonexclusive, irrevocable worldwide license in said article to reproduce, prepare derivative works, distribute copies to the public, and perform publicly and display publicly, by or on behalf of the Government. The Department of Energy will provide public access to these results of federally sponsored research in accordance with the DOE Public Access Plan (\url{http://energy.gov/downloads/doe-public-access-plan}).
\end{document}